\documentclass[11pt]{amsart}

\usepackage[T1]{fontenc}
\usepackage[utf8]{inputenc}
\usepackage{lmodern}
\usepackage{amsmath,amssymb,mathtools}
\usepackage{microtype}
\usepackage[colorlinks=true,
            linkcolor=blue,
            citecolor=blue,
            urlcolor=blue]{hyperref}

\newtheorem{theorem}{Theorem}[section]
\newtheorem{proposition}[theorem]{Proposition}

\newtheorem{corollary}[theorem]{Corollary}
\theoremstyle{remark}

\newcommand{\R}{\mathbb R}
\newcommand{\N}{\mathbb N}
\newcommand{\ip}[2]{\langle #1,#2\rangle}
\newcommand{\norm}[1]{\lVert #1\rVert}
\newcommand{\hnorm}[1]{\lvert #1\rvert_2}
\newcommand{\starnorm}[1]{\lVert #1\rVert_*}

\title[Unbounded Convex Sets with the FPP]{An Unbounded Convex Set with the Fixed Point Property for Nonexpansive Mappings}

\author[Japón]{M. Japón}
\address{Universidad de Sevilla, Spain}
\email{japon@us.es}

\author[Laguna-Ricart]{Javier Laguna-Ricart}
\address{Universidad de Murcia, Spain}
\email{javier.l.r@um.es}

\subjclass[2020]{Primary 47H10, 46B20; Secondary 46B03, 47H09}
\keywords{nonexpansive mapping, fixed point property, unbounded convex set, equivalent renorming, normal structure, localization}
\thanks{The first author was supported by   Grant PID2023-148294NB-I00 funded by MICIU/AEI/ 10.13039/501100011033,   and the  projects “Programa operativo PPIT-FEDER Andalucía 2021-2027 SOL2024-31596 and SOL2024-31708, and María de Maeztu CEX2024-001517-M. The second author was supported by Fundaci\'{o}n S\'{e}neca - ACyT Regi\'{o}n de Murcia project 21955/PI/22, by Agencia Estatal de Investigación (Government of Spain) and ERDF project PID2021-122126NB-C32 and by the University of Murcia through its Predoctoral Contracts Programme under the Research Promotion Plan (“Programa de Contratos Predoctorales del Plan de Fomento de la Investigación de la Universidad de Murcia”)}
\begin{document}

\begin{abstract}
We construct an equivalent norm on $\R\oplus\ell_2$ and a nonempty closed convex unbounded subset having the fixed point property for nonexpansive self-mappings.  The ambient renormed space has the fixed point property on bounded closed convex subsets.  The main point is a geometric localization estimate at a distinguished point of the unbounded set; the fixed point conclusion then follows from the Kirk--Ray localization theorem.
\end{abstract}

\maketitle

\section{Introduction}

Let $C$ be a nonempty subset of a Banach space $X$.  We say that $C$ has the \emph{fixed point property for nonexpansive mappings} if every mapping $T\colon C\to C$ satisfying
\[
  \norm{Tx-Ty}\leq \norm{x-y}\qquad(x,y\in C)
\]
has a fixed point.

For bounded closed convex sets, fixed point theory is governed by compactness and geometric hypotheses such as normal structure.  Kirk's theorem asserts that every nonempty weakly compact convex set with normal structure has the fixed point property \cite{Kirk1965}.  The unbounded situation is different.  Ray proved that a nonempty closed convex subset of a Hilbert space has the fixed point property if and only if it is bounded \cite{Ray1980}; Sine later gave a short proof of the unbounded direction \cite{Sine1987}.  Dom\'inguez Benavides established the analogous failure of the fixed point property for every unbounded closed convex subset of $c_0$ \cite{Dominguez2012}.

The general Banach-space question nevertheless remained open in the literature on unbounded domains \cite{AlghamdiKirkShahzad2014} and was recorded explicitly by Kirk as Question~XV in \cite{Kirk2015}:
Can an unbounded closed convex subset of a Banach space have the fixed point property?

In September 2026, OpenAI claimed to have proved a related longstanding problem, that every reflexive Banach space has the fixed point property for nonexpansive mappings on nonempty bounded closed convex subsets \cite{OpenAI2026}.

We give an affirmative construction.  The ambient linear space is $\R\oplus\ell_2$, equipped with an equivalent non-Hilbertian norm.  The set is a positive affine slice
\[
  C=\bigl\{(s,x):s\geq0,\ x\geq0,\ s+\ip{w}{x}=1\bigr\},
  \qquad w=(2^{-n})_{n\geq1}.
\]
It is unbounded because the weight $w_n$ tends to zero.  On the other hand, the norm is tilted so that, along every sequence escaping to infinity in $C$, a distinguished point $p=(1,0)$ is eventually strictly closer than any fixed point of $C\setminus\{p\}$.  This makes the Kirk--Ray localization sets bounded and reduces every nonexpansive self-map of $C$ to the bounded case.

Our main result is the following.

\begin{theorem}\label{thm:main}
There exist a Banach space $X$ and a nonempty closed convex unbounded subset $C\subseteq X$ such that every nonexpansive mapping $T\colon C\to C$ has a fixed point.  The space $X$ may be chosen to be an equivalent renorming of $\R\oplus\ell_2$.
\end{theorem}

All Banach spaces in this paper are real.
\section{A localization principle}

For a closed convex subset $D$ of a Banach space and $a,b\in D$, put
\begin{equation}\label{eq:G-def}
  G_D(a,b)=\bigl\{z\in D:\norm{z-a}\geq\norm{z-b}\bigr\}.
\end{equation}

We shall use the following classical localization theorem.

\begin{theorem}[Kirk--Ray]\label{thm:KR}
Let $X$ be a Banach space in which every nonempty bounded closed convex subset has the fixed point property.  Let $D\subseteq X$ be closed and convex, and let $T\colon D\to D$ be nonexpansive.  If $G_D(a,Ta)$ is bounded for some $a\in D$, then $T$ has a fixed point in $D$; see \cite[Theorem~2.1]{KirkRay1979}.
\end{theorem}

The following immediate consequence is the form needed below.

\begin{corollary}\label{cor:apex}
Let $X$ have the fixed point property on all nonempty bounded closed convex subsets, and let $D\subseteq X$ be closed and convex.  Suppose that there is $p\in D$ such that $G_D(p,y)$ is bounded for every $y\in D\setminus\{p\}$.  Then $D$ has the fixed point property.
\end{corollary}

\begin{proof}
Let $T\colon D\to D$ be nonexpansive.  If $Tp=p$, there is nothing to prove.  Otherwise $G_D(p,Tp)$ is bounded, and Theorem~\ref{thm:KR} applies.
\end{proof}

\section{The renormed space}

Let
\[
  H=\ell_2(\N;\R),
  \qquad
  H_+=\{x=(x_n)\in H:x_n\geq0\text{ for every }n\},
\]
and
\[
  B_+=\{u\in H_+:\norm{u}_2\leq1\}.
\]
On the vector space $X=\R\oplus H$ define
\begin{equation}\label{eq:star-norm}
  \starnorm{(s,x)}
  =\max\left\{
       \frac45\sqrt{s^2+\norm{x}_2^2},
       \sup_{u\in B_+}\left|\frac{s}{10}+\ip{x}{u}\right|
     \right\}.
\end{equation}
We write
\[
  \hnorm{(s,x)}=\sqrt{s^2+\norm{x}_2^2}
\]
for the Hilbert product norm.

\begin{proposition}\label{prop:norm}
Formula \eqref{eq:star-norm} defines a complete norm on $X$, and
\begin{equation}\label{eq:norm-equivalence}
  \frac45\hnorm{(s,x)}
  \leq \starnorm{(s,x)}
  \leq \frac{\sqrt{101}}{10}\hnorm{(s,x)}
  \qquad((s,x)\in X).
\end{equation}
In particular, $(X,\norm{\cdot}_*)$ is reflexive.
\end{proposition}

\begin{proof}
The second term in \eqref{eq:star-norm} is finite by Cauchy–Schwarz and is the supremum of absolute values of continuous linear functionals; hence it is a seminorm.  Its maximum with the norm $(4/5)\hnorm{\cdot}$ is therefore a norm.  The lower estimate in \eqref{eq:norm-equivalence} is immediate.  For the upper estimate, Cauchy--Schwarz gives
\[
  \sup_{u\in B_+}\left|\frac{s}{10}+\ip{x}{u}\right|
  \leq \frac{|s|}{10}+\norm{x}_2
  \leq \frac{\sqrt{101}}{10}\sqrt{s^2+\norm{x}_2^2}.
\]
Since $4/5<\sqrt{101}/10$, this also dominates the first term in \eqref{eq:star-norm}.  Thus the two norms are equivalent.  Completeness and reflexivity follow from the corresponding properties of the Hilbert product norm.
\end{proof}

\begin{proposition}\label{prop:bounded-fpp}
Every nonempty bounded closed convex subset of $(X,\norm{\cdot}_*)$ has the fixed point property for nonexpansive self-mappings.
\end{proposition}

\begin{proof}
Let $X_2$ denote the same linear space with the Hilbert product norm.  By \eqref{eq:norm-equivalence}, the identity isomorphism between $X_2$ and $(X,\norm{\cdot}_*)$ has distortion at most
\begin{equation}\label{eq:distortion}
  \frac{\sqrt{101}/10}{4/5}=\frac{\sqrt{101}}8<\sqrt2.
\end{equation}
Let $N(\cdot)$ denote Bynum's normal-structure coefficient and let $d(\cdot,\cdot)$ denote Banach--Mazur distance.  The Hilbert value is $N(X_2)=\sqrt2$, and Bynum's stability theorem gives
\[
  N(X_2)\leq d(X_2,X)N(X);
\]
see \cite[Theorem~5]{Bynum1980}.  In view of \eqref{eq:distortion}, it follows that $N(X)>1$, so $(X,\norm{\cdot}_*)$ has normal structure.  Since the space is reflexive by Proposition~\ref{prop:norm}, its bounded closed convex subsets are weakly compact.  Kirk's fixed point theorem \cite{Kirk1965} now gives the conclusion.
\end{proof}

\section{The unbounded convex set}

Let
\[
  w=(2^{-n})_{n\geq1}\in\ell_2
\]
and define
\begin{equation}\label{eq:C-def}
  C=\bigl\{(s,x)\in\R\oplus H:
       s\geq0,\ x\in H_+,\ s+\ip{w}{x}=1
     \bigr\}.
\end{equation}
Set
\begin{equation}\label{eq:p-def}
  p=(1,0)\in C.
\end{equation}

\begin{proposition}\label{prop:C-basic}
The set $C$ is nonempty, closed, convex, and unbounded in $(X,\norm{\cdot}_*)$.
\end{proposition}

\begin{proof}
The positive cone $H_+$ is closed and convex, and the functional
$(s,x)\mapsto s+\ip{w}{x}$ is linear and continuous.  Hence \eqref{eq:C-def} defines a closed convex set; it is nonempty because it contains $p$.

If $(s,x)\in C$, then $x,w\geq0$, so
\begin{equation}\label{eq:s-range}
  0\leq s=1-\ip{w}{x}\leq1.
\end{equation}
Let $(e_n)$ be the standard unit vector basis of $\ell_2$ and put
\[
  q_n=(0,2^n e_n).
\]
Since $\ip{w}{2^n e_n}=1$, one has $q_n\in C$.  Moreover,
\[
  \starnorm{q_n}
  =\max\left\{\frac45\,2^n,
       \sup_{u\in B_+}|\ip{2^n e_n}{u}|
     \right\}
  =2^n.
\]
Thus $C$ is unbounded.
\end{proof}

\section{The localization geometry}

The next proposition is the central point of the construction.

\begin{proposition}\label{prop:G-bounded}
For every $y\in C\setminus\{p\}$, the set
\[
  G_C(p,y)=\bigl\{z\in C:\starnorm{z-y}\leq\starnorm{z-p}\bigr\}
\]
is bounded.
\end{proposition}

\begin{proof}
Fix $y\in C\setminus\{p\}$ and suppose, to the contrary, that $G_C(p,y)$ is unbounded.  Choose
\[
  z_k=(s_k,x_k)\in G_C(p,y)
  \qquad\text{with}\qquad
  \starnorm{z_k}\longrightarrow\infty.
\]
By \eqref{eq:s-range} and the norm equivalence \eqref{eq:norm-equivalence},
\[
  r_k:=\norm{x_k}_2\longrightarrow\infty.
\]
After discarding finitely many terms, assume that $r_k\geq2$.  Put
\begin{equation}\label{eq:dg-def}
  d_k=1-s_k\in[0,1],
  \qquad
  g_k=\frac{x_k}{r_k}.
\end{equation}
Then $g_k\in H_+$ and $\norm{g_k}_2=1$.  Passing to a subsequence, we may suppose that $g_k\rightharpoonup g$ weakly in $\ell_2$.  The cone $H_+$ is weakly closed, so $g\in H_+$.  Since $z_k\in C$,
\[
  \ip{w}{g_k}
  =\frac{\ip{w}{x_k}}{r_k}
  =\frac{d_k}{r_k}
  \longrightarrow0.
\]
It follows that $\ip{w}{g}=0$.  Every coordinate of $w$ is strictly positive and $g\geq0$, hence $g=0$.  Consequently,
\begin{equation}\label{eq:g-weak-zero}
  g_k\rightharpoonup0.
\end{equation}

For each $k$, define
\begin{equation}\label{eq:phi-def}
  \phi_k(s,x)=\frac{s}{10}+\ip{x}{g_k}.
\end{equation}
Since $g_k\in B_+$, the definition of the norm gives
\begin{equation}\label{eq:phi-dominated}
  |\phi_k(s,x)|\leq\starnorm{(s,x)}
  \qquad((s,x)\in X).
\end{equation}

We next compute the norm of $z_k-p=(-d_k,x_k)$.  For $u\in B_+$,
\[
  0\leq\ip{x_k}{u}\leq r_k.
\]
The endpoints are attained at $u=0$ and $u=g_k$.  Therefore
\begin{equation}\label{eq:second-component}
  \sup_{u\in B_+}\left|-\frac{d_k}{10}+\ip{x_k}{u}\right|
  =\max\left\{\frac{d_k}{10},r_k-\frac{d_k}{10}\right\}
  =r_k-\frac{d_k}{10}.
\end{equation}
The last equality uses $r_k\geq2$ and $0\leq d_k\leq1$.  Moreover,
\begin{align}\label{eq:strict-dominance}
  \left(r_k-\frac{d_k}{10}\right)^2
  -\frac{16}{25}(r_k^2+d_k^2)
  &=\frac{36r_k^2-20r_kd_k-63d_k^2}{100}>0.
\end{align}
Indeed, on $[2,\infty)\times[0,1]$ the numerator is increasing in $r$, and at $r=2$ it equals
$144-40d-63d^2\geq41$.  Thus the second term in \eqref{eq:star-norm} strictly dominates the Hilbertian term, and \eqref{eq:phi-def}--\eqref{eq:second-component} yield
\begin{equation}\label{eq:norm-attainment}
  \starnorm{z_k-p}
  =r_k-\frac{d_k}{10}
  =\phi_k(z_k-p).
\end{equation}

Write $y=(b,v)$.  Since $y\in C$,
\begin{equation}\label{eq:one-minus-b}
  1-b=\ip{w}{v}.
\end{equation}
If $1-b=0$, then $v=0$ because $v\geq0$ and every coordinate of $w$ is positive; this would give $y=p$.  Hence
\begin{equation}\label{eq:positive-gap}
  1-b>0.
\end{equation}

Finally, $z_k\in G_C(p,y)$, so \eqref{eq:phi-dominated} and \eqref{eq:norm-attainment} imply
\begin{align*}
  0
  &\geq \starnorm{z_k-y}-\starnorm{z_k-p} \\
  &\geq \phi_k(z_k-y)-\phi_k(z_k-p)
   =\phi_k(p-y).
\end{align*}
On the other hand, by \eqref{eq:g-weak-zero},
\[
  \phi_k(p-y)
  =\frac{1-b}{10}-\ip{v}{g_k}
  \longrightarrow\frac{1-b}{10}>0,
\]
contradicting the preceding inequality.  Therefore $G_C(p,y)$ is bounded.
\end{proof}

\section{Proof of the main theorem}

\begin{proof}[Proof of Theorem~\ref{thm:main}]
Take $X=(\R\oplus\ell_2,\norm{\cdot}_*)$ with the norm \eqref{eq:star-norm}, and let $C$ be given by \eqref{eq:C-def}.  Proposition~\ref{prop:C-basic} shows that $C$ is nonempty, closed, convex, and unbounded.  Proposition~\ref{prop:bounded-fpp} gives the fixed point property on every bounded closed convex subset of $X$, while Proposition~\ref{prop:G-bounded} gives the localization hypothesis at $p$.  Corollary~\ref{cor:apex} therefore implies that $C$ has the fixed point property.
\end{proof}

\section*{AI disclosure}
The authors acknowledge that OpenAI's ChatGPT 5.5 was used in the preparation of this manuscript, specifically to generate a preliminary draft of the proof of the main theorem. The authors then carefully examined the proposed proof and substantially streamlined its arguments before incorporating the revised version into the paper. All statements, proofs, and references included in the final manuscript have therefore been thoroughly reviewed by the authors, who assume full responsibility for any errors or inaccuracies that may remain.



\end{document}